\documentclass[11pt]{amsart}
\usepackage{amssymb,amsmath,amsthm,mathdesign}
\usepackage{pdfsync}
\usepackage{color}
\usepackage[colorlinks]{hyperref}
\usepackage{graphicx,color}

\newtheorem{tm}{Theorem}[section]
\newtheorem{defin}{Definition}[section]

\newtheorem{prop}{Proposition}[section]
\newtheorem{lem}{Lemma}[section]
\newtheorem{coro}{Corollary}[section]
\newtheorem{cond}{Condition}[section]

\numberwithin{equation}{section}

\newcommand{\be}{\begin{eqnarray}}
\newcommand{\ee}{\end{eqnarray}}
\newcommand{\ce}{\begin{eqnarray*}}
\newcommand{\de}{\end{eqnarray*}}
\newcommand{\rd}{{\mathbb R^d}}

\def\R{{\mathbb R}}

\def\a{\alpha}

\def\de{\delta}

\def\E{{\mathbb E}}

\allowdisplaybreaks

\allowdisplaybreaks[1]

\begin{document}

\title{Space-time fractional stochastic partial differential equations with L\'evy Noise}

\author{Xiangqian Meng}
\address{Department of Mathematics, University of Washington at Seattle, Seattle, WA,98105, USA}
\email{xqmeng@uw.edu}
\author{ Erkan Nane}
\address{Department of Mathematics and Statistics, Auburn University, Auburn, AL 36849, USA}
\email{nane@auburn.edu}

\date{}


\maketitle


\begin{abstract}
We consider   non-linear time-fractional  stochastic heat type equation
$$\frac{\partial^\beta u}{\partial t^\beta}+\nu(-\Delta)^{\alpha/2} u=I^{1-\beta}_t \bigg[\int_{\mathbb{R}^d}\sigma(u(t,x),h) \stackrel{\cdot}{\tilde N }(t,x,h)\bigg]$$ and
$$\frac{\partial^\beta u}{\partial t^\beta}+\nu(-\Delta)^{\alpha/2} u=I^{1-\beta}_t \bigg[\int_{\mathbb{R}^d}\sigma(u(t,x),h) \stackrel{\cdot}{N }(t,x,h)\bigg]$$
 in $(d+1)$ dimensions, where  $\alpha\in (0,2]$ and $d<\min\{2,\beta^{-1}\}\a$, $\nu>0$, $\partial^\beta_t$ is the Caputo fractional derivative, $-(-\Delta)^{\alpha/2} $ is the generator of an isotropic stable process, $I^{1-\beta}_t$ is the fractional integral operator, ${N}(t,x)$ are Poisson random measure with  $\tilde{N}(t,x)$  being the compensated Poisson random measure. 
 $\sigma:{\mathbb{R}}\to{\mathbb{R}}$ is  a Lipschitz continuous function.
We prove existence and uniqueness of mild solutions to this equation. Our results  extend the   results in the case of parabolic stochastic partial differential equations obtained in \cite{foondun-khoshnevisan-09, walsh}.  Under the linear growth of $\sigma$, we show that the solution of the time fractional stochastic partial differential equation follows an exponential growth with respect to the time. We also show the nonexistence of the random field solution of both stochastic partial differential equations when $\sigma$ grows faster than linear.
\end{abstract}

{\bf Keywords:} {Time-fractional stochastic partial differential equations; fractional Duhamel's principle; Caputo derivatives; Walsh isometry}

\maketitle

\section{Introduction}

\color{black}

Fractional calculus has received plenty of attention because its wide application in the field of physics, chemistry, finance and etc. In \cite{anomalous diffusion}, we notice that many natural phenomena do not fit into the relatively simple description of diffusion developed by Einstein a century ago, such as the forage of food for animals in the forest, the transport of electrons in amorphous semiconductors in an electric field, the travel times of contaminants in groundwater and the proteins diffuse across cell membranes. Some of these phenomena follow  models like  L\'evy flight, a fractal random walk, or  composed of self-similar jumps. Mathematicians have been aware of fractional derivatives for over 300 years, but, like the Pareto distribution that has no mean value, these derivatives only find their ways into the physical sciences due to the relatively recent observations of anomalous diffusion: see, for example, \cite{beyond Brownian motion, randomwalkguide}.

Stochastic partial differential equations (SPDE) have been studied  in mathematics and various sciences as well; see, for example, Khoshnevisan \cite{khoshnevisan-cbms} for a long list of references. The area of SPDEs is interesting to mathematicians because it contains a lot of hard open problems. However, not much have been done for equations driven by discontinuous noise even though this situation has started to change recently, see, for example, \cite{balan} and references therein.

In this paper we  consider the following two time fractional stochastic partial differential equations (TSPDE),

 \begin{equation}\label{equcompen}
\frac{\partial^\beta u}{\partial t^\beta}+\nu(-\Delta)^{\alpha/2} u=I^{1-\beta}_t \bigg[\int_\rd\sigma(u(t,x),h) \stackrel{\cdot}{\tilde N }(t,x,h)\bigg]
\end{equation}
and
\begin{equation}\label{equnoncompen}
\frac{\partial^\beta u}{\partial t^\beta}+\nu(-\Delta)^{\alpha/2} u=I^{1-\beta}_t \bigg[\int_\rd\sigma(u(t,x),h) \stackrel{\cdot}{N }(t,x,h)\bigg]
\end{equation}
where the initial condition $u_0$ is measurable and bounded,
 $-(-\Delta)^{\alpha/2} $ is the fractional Laplacian with $\alpha\in (0,2]$,  $\stackrel{\cdot}{\tilde{N}}(t,x)$ is a compensated Poisson noise with $x\in \R^d$, and $\dot{N}(t,x)$ is a compensated Poisson noise.  Caputo fractional derivative $\partial ^{\beta}_t $ of order $\beta \in (0,1)$,
  is defined by $$\partial^\beta _t f(t)=\frac{1}{\Gamma (1-\beta)} \frac{d}{dt}\int^t_0 (t-s)^{-\beta}(f(s)-f(0))ds,$$ where $\Gamma(\beta ):=\int^\infty_0 t^{\lambda-1}e^{-t}dt$  is the Gamma function.

The meaning of the above fractional derivative of $f$ at time $t$ depends on the whole history of $f(s)$ on $(0,t)$ with the nearest past affecting the present more. (See \cite{zhenqing}). The fractional diffusion equation $\partial^\beta_t u=\Delta u$ with $0<\beta<1$ has been widely used to model the anomalous diffusion exhibiting subdiffusive behavior due to the particle sticking and trapping phenomena (see e.g.\cite{calculus}).  We will prove the existence  and uniqueness of the mild solution of equation \eqref{equcompen} and \eqref{equnoncompen} under the Lipschitz condition for $\sigma$. We will also discuss the existence of the finite energy solution  and the blow-up and non-existence of the solution for both equations under some  specific conditions. This paper is  an extension of the results  in the  papers \cite{globalOmaba} and \cite{nane-mijena-2014}. Here we consider the fractional time derivative and Poisson type noise. Also, L\'evy noise ${\tilde N }(t,x,h)$ or $N(t,x,h)$ has better modeling characteristics  than white noise in financial engineering\cite{Carr2002}, \cite{Carr2003}, signal detection \cite{signal}, and other areas.  It 
can capture some large moves and unpredictable events.

Let $G_t(x)$ be the fundamental solution of  the fractional heat type equation
\begin{equation}\label{Eq:Green0}
\partial^\beta_tG_t(x)=-\nu(-\Delta)^{\alpha/2}G_t(x).
\end{equation}
$G_t(x)$ is the transition density function of $X(E_t)$, where $X$ is an isotropic $\a$-stable L\'evy process in $\R^d$ and $E_t$ is the first passage time of a $\beta$-stable subordinator $D=\{D_r,\,r\ge0\}$, or the inverse stable subordinator of index $\beta$: see, for example, Bertoin \cite{bertoin} for properties of these processes, Baeumer and Meerschaert \cite{fracCauchy} for more on time fractional diffusion equations, and Meerschaert and Scheffler  \cite{limitCTRW} for properties of the inverse stable subordinator $E_t$.

Let $p_{{X(s)}}(x)$ and $f_{E_t}(s)$ be the density of $X(s)$ and $E_t$, respectively. Then the Fourier transform of $p_{{X(s)}}(x)$ is given by
\begin{equation}\label{Eq:F_pX}
\widehat{p_{X(s)}}(\xi)=e^{-s\nu|\xi|^\a},
\end{equation}
and
\begin{equation}\label{Etdens0}
f_{E_t}(x)=t\beta^{-1}x^{-1-1/\beta}g_\beta(tx^{-1/\beta}),
\end{equation}
where $g_\beta(\cdot)$ is the density function of $D_1.$ The function $g_\beta(u)$ (cf. Meerschaert and Straka \cite{meerschaert-straka}) is infinitely differentiable on the entire real line, with $g_\beta(u)=0$ for $u\le 0$.

By conditioning, we have
\begin{equation}\label{Eq:Green1}
G_t(x)=\int_{0}^\infty p_{_{X(s)}}(x) f_{E_t}(s)ds.
\end{equation}
We define a Poisson random measure (\textit{or non-compensated Poisson random measure}), $N:=\Sigma_{i\geq1}\delta_{(T_i,X_i,Z_i)}$ on $\R_{+}\times{\R^d} \times{\R^d} $ defined on a probability space$ (\Omega,\mathcal{F},P) $ with intensity measure
$\mathrm{d}t \mathrm{d}x \mu(\mathrm{d}h)$. Throughout  this paper we assume that $\mu$ is a L\'evy measure on $\R^d$, which satisfies the following
$$
\int_{\R^d}(1\wedge {|h|^2}) \mu(\mathrm{d}h)<\infty.
$$
Then we set $\tilde{N}(\mathrm{d}s\mathrm{d}x\mathrm{d}h)=N(\mathrm{d}s\mathrm{d}x\mathrm{d}h)-\mathrm{d}s\mathrm{d}x\mu(\mathrm{d}h)$ and  call $\tilde{N}$ the \textit{compensated Poisson Random measure}.
In this paper we study the existence and uniqueness of the solution to \eqref{equcompen} under global Lipchitz conditions on $\sigma$, using the white noise  approach of \cite{walsh}.

We say that a random field $\{u(t,x)\}_{x \in \R^d, t>0}$ is a mild solution of equation \eqref{equcompen} if a.s., the following is satisfied
$$u(t,x)=\int_{\R^d}G(t,x-y){u_0}(y)\mathrm{d}y+\int_0^t \int_{\R^d}\int_{\R^d} G(t-s,x-y)\sigma(u(s,y),h)\tilde{N}(\mathrm{d}s\mathrm{d}x\mathrm{d}h).$$

For more explanation of mild solutions about Cauchy problem, please refer to \cite{Wolfgang}. We also refer to Mijena and Nane \cite{nane-mijena-2014} for the use of time fractional Duhamel's principle in obtaining the mild solutions.

For the existence and uniqueness of solutions to \eqref{equcompen} we  need the following condition on $\sigma$.\\
\begin{cond}\label{cond-compen}
 There exists a non-negative function $J$ and a finite positive constant $Lip_\sigma$, such that for all $x,y,h \in \R^d$, we have
$$\sigma(0,h)=0\hbox{  and } |\sigma(x,h)-\sigma(y,h)| \leq J(h)Lip_\sigma |x-y|.$$The function $J$ is assumed to satisfy the following integrability condition,
 $\int_{\R^d} J(h)^2\mu(dh)\leq K$,
 where $K$ is some positive finite constant.
\end{cond}
For the existence and uniqueness of solutions to \eqref{equnoncompen} we need the following condition on $\sigma$.
\begin{cond}\label{condofnoncompen}
 There exists a non-negative function $J$ and a finite positive constant $Lip_\sigma$, such that for all $x,y,h \in {\R^d}$, we have
$$\sigma(0,h)=0 \hbox{ and }|\sigma(x,h)-\sigma(y,h)| \leq J(h)Lip_\sigma |x-y|.$$
 The function $J$ is assumed to satisfy the following integrability condition,
 $\int_{\R^d} J(h)\mu(dh)\leq K$,
 where $K$ is some positive finite constant.
\end{cond}

{ The fractional integral of the noise term  in equations \eqref{equcompen} and \eqref{equnoncompen} are  not merely used to get a simple integral solution. A physical important reason to take the fractional integral of the noise in  these equations:
Apply the fractional derivative of order $1-\beta$ to both sides of these equations to see the forcing function, in the traditional units $x/t$: see, for example, Meerschaert et al  \cite{Meerschaert-2016}. In this paper the authors work on a deterministic time fractional equation with an external force, but the same physical principle should apply for the stochastic equations too.
}

We now briefly give an outline of this paper. We adapt the methods of  proofs of the results in  \cite{nane-mijena-2014} with  many crucial nontrivial changes. We state the main results of the paper in Section \ref{sec:2}. We give some preliminary results in Section \ref{sec:3}. Moment estimates for time increments and spatial increments of the solution are given in  \ref{sec:4}.  The main result in this section is Proposition \ref{stoyoungcompen} and Proposition \ref{stochyoungnoncompen} under Lipschitz conditions of $\sigma$.
In Sections \ref{sec:5} and  \ref{sec:6}, we prove the main results of the paper  under some  conditions of $\sigma$. We also give the behavior of the growth of the moments of the solutions when $\sigma$ is growing linearly. In addition, we  also show that under faster than linear growth of $\sigma$, there is  no finite energy solution for equation \eqref{equcompen} with the compensated Poisson noise, and no random field solution for both equations.
\section{Statement of main results}\label{sec:2}

Our first existence and uniqueness result is the following theorem.
\begin{tm}\label{Existcompen}
 Let $d < \min\{2,\gamma^{-1}\}\alpha$. If $u_0$ is measurable and bounded,  then there exists a unique random field solution  to equation \eqref{equcompen} under Condition \ref{cond-compen}.
\end{tm}
Next, we show that a result of the growth of the second moment of  the solution to equation \eqref{equcompen} under condition of the linear growth of $\sigma$.
\begin{cond}\label{condof-beyondlinear}
There exists a positive function $\overline {J}$ and a constant L such that for all $x,h \in \R^d$ , we have
$$|\sigma(x,h)|\geq L \overline{J}(h)|x|.$$
The function $\overline{J}$ is assumed to satisfy the following integrability condition,
$$\int_{\R^d} \overline{J}(h)^2 \mu(\mathrm{d}h)\geq \kappa,$$
where $\kappa$ is some positive constant.
\end{cond}

The next result proves the intermittency property  of the solution of equation \eqref{equcompen}.

\begin{tm}\label{intermittency-compen}
Let $d < \min\{2,\gamma^{-1}\}\alpha$. Suppose that Condition \ref{cond-compen} holds and  $u_0$ is bounded above, which means there is a positive number $\eta_1$ such that $\eta_1:=\sup_{x\in \R^d} u_0(x)$, then the solution $u$ of equation  \eqref{equcompen} satisfies
$$
  \sup_{x \in \R^d}\E|u(t,x)|^2\leq c_1e^{c_2 t} \hbox{ for all } t>0.
  $$
{ where $c_1:= \eta_1^2$ and $c_2$ depends on $\alpha,\beta,d,\kappa$ and $Lip_\sigma$ in Condition \ref{cond-compen}.}\\
Similarly, if Condition  \ref{condof-beyondlinear} holds and $u_0$ is bounded below,  which means there is a positive number $\eta_2$ such that $\eta_2:=\inf_{x\in \R^d}u_0(x)$, then the solution $u$ of  equation \eqref{equcompen} satisfies
 $$
  \inf_{x \in \R^d}\E|u(t,x)|^2\geq c_1e^{c_2 t} \hbox{ for all } t>0.$$
  where $c_1:= \eta_2^2$ and $c_2$ depends on $\alpha,\beta,d,\kappa$ and $L$ in Condition \ref{condof-beyondlinear}.
\end{tm}

{If $\{X_t\}_{t\geq 0}$ is a L\'evy process  with characteristic function  $\E\exp(i\xi X_t)=\exp(-t\psi(\xi))$, for all $t>0$,  $\xi\in \R^d$ where $\psi(\xi)$ is the characteristic exponent. When  $\{X_t\}_{t\geq 0}$ is a symmetric $\alpha$-stable process,  the characteristic exponent is $\psi(\xi)=\nu|\xi|^\alpha$, corresponding to the fractional Laplacian generator $-\nu(-\Delta)^{\alpha/2}$.
Define $\Upsilon(\gamma):=\frac{1}{2\pi}\int_{\R^d}\frac{d\xi}{\gamma+2Re\psi(\xi)} \hbox{ for all } \gamma \in (0,\infty).$

\begin{tm}\label{lowerbound}
Suppose that the assumption of  Theorem \ref{Existcompen} are in force, $\beta=1$, and that Condition \ref{condof-beyondlinear} holds. If the initial function is bounded below, then
$$
\overline{\gamma} (2)\geq \Upsilon^{-1}(\frac{1}{\kappa L^2})>0,
$$
where $\Upsilon^{-1}(t):=\sup\{\lambda>0:\Upsilon(\lambda)>t\} $.

\end{tm}
}
This theorem is an extension of the corresponding result in \cite{foondun-khoshnevisan-09} to SPDEs with L\'evy noise.
\color{black}
We will establish the non-existence of finite energy solutions when $\sigma$ grows faster than linear.\\
A random field $u$ is a \textit{finite energy solution} to the fractional stochastic heat equation \eqref{equcompen} when $u \in \cup_{\gamma>0} \mathcal{L}^{\gamma,2} $  and there exists $\rho_{*}>0$ such that
$$\int_0^\infty e^{-\rho_{*}t}\E(|u_t(x)|^2)\mathrm{d}t<\infty \text{ for all } x \in \R^d.
$$
\begin{cond}\label{condoflip}
There exist  constants  $L ,\rho$ and a positive function $\overline{J}$  such that $\rho>1$ and for all $x,h\in {\R^d}$, we have
$$|\sigma(x,h)|\geq L \overline{J} (h)|x|^\rho,$$
where the function $\overline{J}$ is the same as in Condition \ref{condof-beyondlinear}.
\end{cond}

\begin{tm}\label{nonexistcompen}
Let let $d < \min\{2,\gamma^{-1}\}\alpha$. If the initial condition $u_0(x)$ is bounded below, then there is no finite energy solution to equation \eqref{equcompen} under Condition \ref{condoflip}.
\end{tm}

 Under the same condition, there is also no random field solution to equation (\ref{equcompen}).

\begin{tm}\label{nonexistofcompen}
Let $d < \min\{2,\gamma^{-1}\}\alpha$. If the initial condition $u_0(x)$ is bounded below, then there is no random field solution to equation\eqref{equcompen} under Condition \ref{condoflip} .
\end{tm}

{ Bao and Yuan \cite{BaoYuan2016} studied the finite time blow-up in $L^p$-norm of stochastic reaction-diffusion equations with jumps within a bounded domain.    Li et al. \cite{LiPengJia2016} considered the blow-up in $L^p$-norm for a class of L\'evy noise driven SPDEs.
}

In the remainder of the section  we will state some  properties of  the solution to equation \eqref{equnoncompen}.
First, we will present a similar theorem about the existence and uniqueness of the equation \eqref{equnoncompen}.
\begin{tm}\label{existofnoncompen}
There exists a unique random field solution to the Equation \eqref{equnoncompen} under Condition \ref{condofnoncompen}.
\end{tm}
If the growth of $\sigma$ is linear, the next theorem shows that the solution of \eqref{equnoncompen} grows exponentially.
\begin {cond}\label{condofnoncompen1}
There exists a positive function $\overline{J}$ and a constant L such that for all $x, h\in \R^d$, we have
$$
|\sigma(x,h)|\geq L \overline{J}(h)|x|.
$$
The function $\overline{J}$ is assumed to satisfy the following integrability condition $\int_{\R^d} \overline J(h)dh\geq \kappa$, where $\kappa $ is some positive constant.

\end{cond}

\begin{tm}\label{intermittency-noncomp}
 Suppose  that Condition \ref{condofnoncompen} and Condition \ref{condofnoncompen1} hold. If the initial condition $u_0(x)$ is bounded below, then the solution to equation \eqref{equnoncompen} satisfies
 $$
\inf_{x\in \R^d} \E|u(t,x)|\geq c_1e^{c_2 t}  \hbox{  for all }  t>0.
$$
where $c_1:= \inf_{x \in R^d}|u(0,x)|$ and $c_2$ depends on $\alpha,  \beta, d, \kappa $ and $L$ in Condition  \ref{condofnoncompen1}.

\end{tm}

\begin{cond}\label{condoflipnoncompen2}
There exists a constant  $L ,\lambda$ and a positive function $\overline{J}$, such that $\lambda>1$ and for all $x,h\in {\R^d},$ we have
$$|\sigma(x,h)|\geq L \overline{J}(h) |x|^\lambda,$$

where the function $\overline{J}$ are the same as in Condition \ref{condofnoncompen1}.

\end{cond}

\begin{tm}\label{no-rf-sol-noncom}If the initial condition $u_0(x)$ is bounded below, then there is no random field solution to equation\eqref{equnoncompen} under Condition \ref{condoflipnoncompen2}.
\end{tm}





\section{Preliminaries}\label{sec:3}

In this section, we give some preliminary results that will be needed in the remaining sections of the paper.

We first have the following lemma from \cite{nane-mijena-2014}.

\begin{lem}[Lemma 1 in Mijena and Nane \cite{nane-mijena-2014} ]\label{Lem:Green1} For $d < 2\a,$
\begin{equation}\label{Eq:Greenint}
\int_{{\R^d}}G^2_t(x)dx  =C^\ast t^{-\beta d/\a}
\end{equation}
where $C^\ast = \frac{(\nu )^{-d/\a}2\pi^{d/2}}{\a\Gamma(\frac d2)}\frac{1}{(2\pi)^d}\int_0^\infty z^{d/\a-1} (E_\beta(-z))^2 dz.$
\end{lem}

Next we define the stochastic integrals with respect to Poisson random measures by  giving the definition of simple random field.\\

Using the filtration of $\mathcal{F}_t=\sigma\{N(t), t\geq 0\}$, we say a random field $\Phi$ is\textit{ elementary} if it has the following representation:

$$\Phi_t(x,h)=X1_{(a,b]}(t)\phi(x,h) \quad \hbox{for all } t\geq 0 \hbox{ and }  x\in \R^d$$
 for some $0\leq a \leq b< \infty $, where $ X$ is a $\mathcal{F}_a $ measurable random variable in $L^2(\Omega)$, and $\phi$ is non-random, bounded and measurable. It is natural to define the stochastic integral$$\int h\Phi N(\mathrm{d}s\mathrm{d}x\mathrm{d}h)=X\int_{(a,b]
\times \R^d\times \R^d}h_t(x,h)\phi(x,h) N(\mathrm{d}s\mathrm{d}x\mathrm{d}h)$$
for any $h\in L^2(\R_+, \R^d,\R^d)$.\\

A random field $\Phi$ is \textit{simple} if there exist elementary random fields $\Phi^{(1)},...,\Phi^{(n)}$ with disjoint support  such that $\int h\Phi \mathrm{d}N=\Sigma_{i=1}^{n}\int h\Phi^{(i)}\mathrm{d}N$.\\ (See \cite{khoshnevisan-cbms} for detailed definition based on the white noise.)

With this notion, we have the following sequence of definitions.
\begin{defin}
Suppose that $X(s,x, h)$ is a predictable process such that
$$ \int_0^T\int_{\R^d}\int_{\R^d} \E|X(s,x,h)|\mathrm{d}s\mathrm{d}x\mu(\mathrm{d}h)<\infty,$$
then
$$ \int_0^T\int_{\R^d}\int_{\R^d} X(s,x,h)N(\mathrm{d}s\mathrm{d}x\mathrm{d}h)$$ is well-defined and satisfies the following  isometry
\begin{equation}\label{isometrynon}
\begin{split}
 &\E \int_0^T\int_{\R^d}\int_{\R^d} X(s,x,h)N(\mathrm{d}s\mathrm{d}x\mathrm{d}h)\\
 &=\int_0^T\int_{\R^d}\int_{\R^d} \E X(s,x,h)\mathrm{d}s\mathrm{d}x\mu(\mathrm{d}h).
 \end{split}
 \end{equation}
 \end{defin}

  We then have the following definition.

 \begin{defin}
Suppose that $X(s,x, h)$ is a predictable process such that
$$ \int_0^T\int_{\R^d}\int_{\R^d} \E|X(s,x,h)|^2\mathrm{d}s\mathrm{d}x\mu(\mathrm{d}h)<\infty,$$
then
$$ \int_0^T\int_{\R^d}\int_{\R^d} X(s,x,h)\tilde{N}(\mathrm{d}s\mathrm{d}x\mathrm{d}h)$$ is well-defined and satisfies the following  isometry
\begin{equation}\label{isometry}
\begin{split}
 &\E \bigg|\int_0^T\int_{\R^d}\int_{\R^d} X(s,x,h)\tilde{N}(\mathrm{d}s\mathrm{d}x\mathrm{d}h)\bigg|^2\\
 &=\int_0^T\int_{\R^d}\int_{\R^d} \E X^2(s,x,h)\mathrm{d}s\mathrm{d}x\mu(\mathrm{d}h).
 \end{split}
 \end{equation}
 \end{defin}

We are now ready to state the precise meaning of the solutions. We define the \textit{mild solution} of equation \eqref{equcompen} first.
\begin{defin}
We say that a random field $\{u(t,x)\}_{t>0,x \in \R^d}$ is a mild solution of equation \eqref{equcompen} if a.s., the following is satisfied
$$u(t,x)=\int_{\R^d}G(t,x-y){u_0}(y)\mathrm{d}y+\int_0^t \int_{\R^d}\int_{\R^d} G(t-s,x-y)\sigma(u(s,y),h)\tilde{N}(\mathrm{d}s\mathrm{d}x\mathrm{d}h).$$
\end{defin}
Since we are mainly  interested in the second moment of the solution of equation \eqref{equcompen}, and we say that if $\{u(t,x)\}_{t>0,x \in \R^d}$ satisfies the following condition
$$
\sup_{0\leq t\leq T}\sup_{x\in \R^d} \E|u(t,x)|^2<\infty.
$$
for all $T> 0$, then $\{u(t,x)\}_{t>0,x \in \R^d}$ is a $random$ $field$ $ solution$ to the equation \eqref{equcompen}.

For any $\gamma>0$, define the following norm
$$||u||_{2,\gamma}:=\{\mbox{sup}_{t>0}\mbox{sup}_{x \in \R^{d}}e^{-\gamma t}\E[|u(t,x)|^2]\}^{1/2}.$$

\begin{defin}\label{comp-space}
We denote by $\mathcal{L}^{\gamma, 2}$  the completion of the space of all simple  random  field in the norm $\||u||_{2,\gamma}.$

\end{defin}

Switching to the solution of equation \eqref{equnoncompen}, we similarly define the \textit{mild solution} of equation  \eqref{equnoncompen}.
\begin{defin}
We say that a random field $\{u(t,x)\}_{t>0,x \in \R^d}$ is a mild solution of equation  \eqref{equnoncompen} if a.s.,the following is satisfied
$$u(t,x)=\int_{\R^d}G(t,x-y){u_0}(y)\mathrm{d}y+\int_0^t \int_{\R^d}\int_{\R^d} G(t-s,x-y)\sigma(u(s,y,h)N(\mathrm{d}s\mathrm{d}x\mathrm{d}h).$$
\end{defin}
We are interested in the first moment of the solution of equation\eqref{equnoncompen}, and we say that if $\{u(t,x)\}_{t>0,x \in \R^d}$ satisfies the following condition
$$
\sup_{0\leq t\leq T}\sup_{x\in \R^d} \E|u(t,x)|<\infty.
$$
for all $T> 0$, then $\{u(t,x)\}_{t>0,x \in \R^d}$ is a $random$ $field$ $ solution$ to the equation \eqref{equnoncompen}.

Define the following norm  for any $ \gamma>0$
$$||u||_{1,\gamma}:=\{\mbox{sup}_{t>0}\mbox{sup}_{x \in \R^{d}}e^{-\gamma t}\E[|u(t,x)|]\}.$$

\begin{defin}\label{comp-space1}
We denote by $\mathcal{L}^{\gamma, 1}$  the completion of the space of all simple  random fields in the norm $\||u||_{1,\gamma}.$
\end{defin}

We also quote the following propositions that will be needed in the proof of our main results.

\begin{prop}[Proposition 2.12 in  Foondun and Nane \cite{mohammud-nane}]\label{expestimateupper}
Let $\rho>0$ and suppose $f(t)$ is nonnegative, locally integrable functions satisfying
$$
f(t)\leq c_1+\kappa^{'}\int_0^t (t-s)^{\rho-1}f(s)\mathrm{d}s,\text{ for all }t>0,
$$
where $c_1$ is some positive number. Then, we have
$$
f(t)\leq c_2\mathrm{exp}(c_3(\Gamma(\rho))^{1/\rho}(\kappa^{'})^{1/\rho}t),\text{ for all }t>0.
$$
for some positive constants $c_2$ and $c_3$.

\end{prop}

\begin{prop}[Proposition 2.13 in  Foondun and Nane \cite{mohammud-nane}]\label{expestimate}
Let $\rho>0$ and suppose $f(t)$ is nonnegative, locally integrable functions satisfying
$$
f(t)\geq c_1+\kappa^{'}\int_0^t (t-s)^{\rho-1}f(s)\mathrm{d}s,\text{ for all }t>0,
$$
where $c_1$ is some positive number. Then, we have
$$
f(t)\geq c_2\mathrm{exp}(c_3(\Gamma(\rho))^{1/\rho}(\kappa^{'})^{1/\rho}t),\text{ for all }t>0.
$$
for some positive constants $c_2$ and $c_3$.
\end{prop}

\begin{prop}[Proposition 2.12 in Asogwa et al. \cite{asogwa-foondun-mijena-nane-2017}]\label{nonlinearinequ}
Let $\theta>0$. Suppose $h$ is a non-negative function satisfying the following non-linear integral inequality,
$$
h(t)\geq C+D\int_0^t\frac{h(s)^{1+\gamma}}{(t-s)^{\theta}}\,\mathrm{d} s,\quad \text{for}\quad t>0
$$
where $C$, $D$ and $\gamma$ are positive numbers. Then for any $C>0$ there exists $t_0>0$ such that $h(t)=\infty$ for all $t\geq t_0$.
\end{prop}

\color{black}
\section{Estimates on moments of the increments of the solution}\label{sec:4}

In this section we are going to prove a stochastic Young's inequality for both compensated Poisson integrals and non-compensated Poisson integrals. These inequalities are very crucial for proving our main theorems on existence and uniqueness of solutions.

Now we set $(P_tu_0)(x):=\int_{\R^d} G(t,x-y)u_0(y)\mathrm{d}y$ and
$$(\mathcal{A}u)(t,x):=\int_0^t \int_{\R^d}\int_{\R^d} G(t-s,y-x)\sigma(u(s,y),h) \tilde{N}(\mathrm{d}s\mathrm{d}y\mathrm{d}h).$$
Next, we prove a stochastic Young's Inequality for equation \eqref{equcompen} with compensated Poisson noise.
\begin{prop}\label{stoyoungcompen}
Suppose that $ d<min\{2,\beta^{-1}\} \alpha$, Condition \ref{cond-compen} holds,  $u$ admits a predictable random field solution of equation \eqref{equcompen} and $||u||_{2,\gamma}<\infty$ for $\gamma >0$. Then there exists some positive constant $C^{**} $ such that
$||\mathcal{A}u||_{2,\gamma}\leq C^{**}||u||_{2,\gamma}$, where $C^{**}=Lip_\sigma\sqrt{\frac{{C^*K}\Gamma(1-\frac{\beta d}{\alpha})}{\gamma^{1+\frac{\beta d}{\alpha}}}}$.
\end{prop}
\begin{proof}
By applying  the Walsh isometry for compensated Poisson integrals  in equation \eqref{isometry}, Lemma \ref{Lem:Green1}, and Condition \ref{cond-compen} we get
\begin{eqnarray*}
\E|\mathcal{A}u(t,x)|^2&=&\int_0^t\int_\rd\int_\rd |G(t-s,y-x)|^2 \E|\sigma(u(s,y),h)|^2 \mathrm{d} s\mathrm{d} y\mu(\mathrm{d}h)\\
& \leq&\int_0^t\int_\rd\int_\rd |G(t-s,y-x)|^2J(h)^2Lip_\sigma^2\E|u(s,y)|^2 \mathrm{d}s \mathrm{d} y\mu(\mathrm{d}h)\\
&\leq&\int_0^t\int_\rd|G(t-s,y-x)|^2\bigg[\int_\rd J(h)^2\mu(\mathrm{d}h)\bigg]\\
&&\times Lip_\sigma^2 e^{\gamma s}\bigg[\sup_{s>0}\sup_{y\in \R^d} e^{-\gamma s}E|u(s,y)|^2 \bigg]\mathrm{d}s \mathrm{d}y\\
&\leq&KLip_\sigma^2||u||_{2,\gamma}^2  \int_0^t\int_\rd |G(t-s,y-x)|^2 e^{\gamma s}\mathrm{d} y\mathrm{d}s\\
&=&KLip_\sigma^2e^{\gamma t}||u||_{2,\gamma}^2 \int_0^t C^*(t-s)^{-\beta d/\alpha}e^{-\gamma (t-s)}\mathrm{d}s\\
&\leq&KLip_\sigma^2 e^{\gamma t}||u||_{2,\gamma}^2 \int_0^\infty C^*(s)^{-\beta d/\alpha}  e^{-\gamma s}\mathrm{d}s.
\end{eqnarray*}

Let $r=\gamma s$, then
\begin{eqnarray*}
\E|\mathcal{A}u(t,x)|^2 &\leq& e^{\gamma t}KLip_\sigma^2||u||_{2,\gamma}^2 \int_0^\infty C^*(\frac{r}{\gamma})^{-\beta d/\alpha} e^{-r} \frac{1}{\gamma}\mathrm{d}r\\
&=& e^{\gamma t}KLip_\sigma^2||u||_{2,\gamma}^2 \int_0^\infty C^*\frac{r^{-\beta d/\alpha}}{\gamma^{1+\beta d/\alpha}} e^{-r} \mathrm{d}r\\
&=&\frac{e^{\gamma t}C^*KLip_\sigma^2||u||_{2,\gamma}^2}{\gamma^{1+\beta d/\alpha}} \Gamma(1-\beta d/\alpha).
\end{eqnarray*}

Therefore $||\mathcal{A}u||_{2,\gamma}\leq C^{**}||u||_{2,\gamma}$, where $C^{**}=Lip_\sigma\sqrt{\frac{{C^*K}\Gamma(1-\frac{\beta d}{\alpha})}{\gamma^{1+\frac{\beta d}{\alpha}}}}$.
\end {proof}

\color{red}

\color{black}
We then present a corollary of the above stochastic Young's inequality by some substitutions.
\begin{coro}\label{stoyoungcompen1}
Suppose that $d<\min\{2,\beta^{-1}\} \alpha$,$\gamma>0$, and Condition \ref{cond-compen} holds. For any predictable random field solutions $u$ and $v$  of \eqref{equcompen} satisfying $||u||_{2,\gamma}+||v||_{2,\gamma}<\infty,$ we have
$||\mathcal{A}u-\mathcal{A}v||_{2,\gamma}\leq C^{**}||u-v||_{2,\gamma}$, where $C^{**}=Lip_\sigma \sqrt{\frac{{C^*K}\Gamma(1-\frac{\beta d}{\alpha})}{\gamma^{1+\frac{\beta d}{\alpha}}}}$.

\end{coro}

\begin{proof}
Applying  the proof of Proposition \ref{stoyoungcompen} to $u-v$, we have

\begin{eqnarray*}
\E|\mathcal{A}u(t,x)-\mathcal{A}v(t,x)|^2& =&\int_0^t\int_{\R^d}\int_{\R^d} |G(t-s,y-x)|^2\\
 &&\times \E|\sigma(u(s,y),h)-\sigma(v(s,y),h)|^2\mathrm{d}s\mathrm{d}y\mu(\mathrm{d}h)\\
&\leq& \frac{C^*e^{\gamma t}||u-v||_{2,\gamma}^2 KLip_\sigma^2}{\gamma^{1+\beta d/\alpha}} \Gamma(1-\beta d/\alpha).
\end{eqnarray*}

Take the square root,  it yields $||\mathcal{A}u-\mathcal{A}v||_{2,\gamma}\leq C^{**}||u-v||_{2,\gamma}$, where $C^{**}=Lip_\sigma \sqrt{\frac{{C^*K}\Gamma(1-\frac{\beta d}{\alpha})}{\gamma^{1+\frac{\beta d}{\alpha}}}}$.

\end{proof}

There is also a corresponding stochastic Young's inequality for equation \eqref{equnoncompen} with non-compensated Poisson noise. We define the following operator first
$$
\mathcal{B}u(t,x):=\int_0^t \int_{\R^d}\int_{\R^d} G(t-s,y-x)\sigma(u(s,y),h) {N}(\mathrm{d}s\mathrm{d}y\mathrm{d}h).
$$
\begin{prop}\label{stochyoungnoncompen}

Suppose that u is predictable random field solution of equation \eqref{equnoncompen} such that $||u||_{1,\gamma} <\infty$, then under Condition \ref{condofnoncompen}, we have
$$||\mathcal{B}u(t,x)||_{1,\gamma}\leq \frac{K Lip_{\sigma}}{\gamma} ||u||_{1,\gamma}.$$
\end{prop}

\begin{proof}
Using the isometry \eqref{isometrynon} we have
\begin{eqnarray*}\label{stoyoungnoncom}
\E|\mathcal{B}u(t,x)|&=&\int_0^t\int_{\R^d}\int_{\R^d}|G(t-s,y-x)|\E|\sigma(u(s,y),h)|\mu(\mathrm{d}h)\mathrm{d}y\mathrm{d}s\\
&\leq& \int_0^t\int_{\R^d}\int_{\R^d}|G(t-s,y-x)|J(h)Lip_{\sigma}\E|u(s,y)|\mu(\mathrm{d}h)\mathrm{d}y\mathrm{d}s\\
&\leq& KLip_{\sigma}\int_0^t\int_{\R^d}|G(t-s,y-x)|\E|u(s,y)|\mathrm{d}y\mathrm{d}s.\\
\end{eqnarray*}
Mutiply both side by $e^{-\gamma t}$, it yields
\begin{eqnarray*}\label{stoyoungnoncom1}
e^{-\gamma t}\E|\mathcal{B}u(t,x)|&\leq &KLip_{\sigma}\int_0^t\int_{\R^d} e^{-\gamma(t-s)}|G(t-s,y-x)|e^{-\gamma s}\E|u(s,y)|\mathrm{d}y\mathrm{d}s\\
&\leq& KLip_{\sigma}||u||_{1,\gamma}\int_0^\infty\int_{\R^d} e^{-\gamma (t-s)}|G(t-s,y-x)|\mathrm{d}y\mathrm{d}s\\
&\leq&K Lip_{\sigma}||u||_{1,\gamma}\int_0^\infty e^{-\gamma s}\mathrm{d}s\\
&=&\frac{K Lip_{\sigma}}{\gamma} ||u||_{1,\gamma}\\
\end{eqnarray*}
\end{proof}

We also obtain a corollary of the above stochastic Young's inequality for Equation \eqref{equnoncompen} with non-compensated Poisson noise.
\begin{coro}\label{stochyoungnoncompen2}
Suppose that 
 Condition \ref{condofnoncompen} holds. For any predictable random field solutions $u$ and $v$ of \eqref{equnoncompen} satisfying $||u||_{1,\gamma}+||v||_{1,\gamma}< \infty$ for all $\gamma> 0$, we have
$$
||\mathcal{B}u-\mathcal{B}v||_{1,\gamma}\leq \frac{K Lip_{\sigma}}{\gamma}||u-v||_{1,\gamma} .
$$
\end{coro}
\begin{proof}
The proof is  similar to the proof of  Corollary \ref{stoyoungcompen1}.
\end{proof}

\section{Proof of  results for the compensated Poisson noise}\label{sec:5}

\begin{proof}[Proof of Theorem \ref{Existcompen}]
We use Picard iteration  to show existence of solutions.
Let $u^{(0)}
(t,x) := u_0(x)$, and for all $n\geq 0$
$$
u^{(n+1)}(t,x)=(P_{t}u_0)(x)+(\mathcal{A}u^{(n)})(t,x)
$$
and
$$
(P_{t}u_0) (x)=\int_{\R^d} G_t (x-y) u_0 (y)dy.
$$
Then we have
$$
||P_{t}u_0 (x)||_{2,\gamma}=||u(0,x)||_{2,\gamma}=\{\sup_{t>0}\sup_{x \in R^d} e^{-\gamma t}E|u(0,x)|^2\}^{\frac{1}{2}}\leq \sup _{x\in \R^d}|u_0 (x)|.
$$

Since by Proposition \ref{stoyoungcompen}, $||\mathcal{A}u||_{2,\gamma} \leq C^{**}||u||_{2,\gamma}, \text {for } \gamma >0$
it yields

$$||u^{(n+1)}||_{2,\gamma} \leq \sup_{x \in \R^d}|u_0(x)|+C^{**}||u^{(n)}||_{2,\gamma}.$$

We can find a constant $L_2$ depending only on $Lip_\sigma,\alpha, \beta,d$  such that  $\gamma:=(L_2 K)^\frac{\alpha}{\alpha+\beta d}$ satisfying

$$C^{**}=Lip_\sigma\sqrt{\frac{{C^*K}\Gamma(1-\frac{\beta d}{\alpha})}{\gamma^{1+\frac{\beta d}{\alpha}}}}\leq\frac{1}{4}.$$

and let $\theta=\sup_{x\in \R^d}|u_0(x)|$

then
\begin{eqnarray*}
||u^{(n+1)}||_{2,\gamma}&=&||(P_{t}u_0)(x)+(\mathcal{A}u^{(n)})(t,x)||_{2,\gamma}\\
&\leq& \sup_{x\in \R^d}|u_0(x)|+\frac{1}{4}||u^{(n)}||_{2,\gamma} \\
&\leq&\theta+\frac{1}{4}(\theta+||u^{(n)}||_{2,\gamma})\\
&\leq&\sum_{j=0}^{n} \frac{\theta}{4^j} +\frac{1}{4^{n+1}}\sup_{x\in \R^d}||u^{(0)}(x)||\\
&=&\frac{\theta}{1-\frac{1}{4}}+\frac{1}{4^{n+1}}\theta\\
&\leq& \frac{4}{3}\theta+\theta=\frac{7}{3}\theta=L_1.\\
\end{eqnarray*}
Hence, $||u^{(n+1)}||_{2,\gamma} \in\cup_{\gamma>0} \mathcal{L}^{\gamma,2},$
i.e. $$\sup_{t>0}\sup_{x\in \R^d} e^{-2\gamma t}\E|u^{(n+1)}(t,x)|^2\leq (L_1)^2.$$
Therefore,we obtain
$$\sup_{x\in \R^d} \E|u^{(n+1)}(t,x)|^2\leq (L_1)^2 e^{2\gamma t}, \hbox{for all }  t>0.$$

Next, we use the Banach fixed point theorem to show the existence of the solution.
By the stochastic Young's inequality, we have
\begin{eqnarray*}
||u^{(n+1)}-u^{(n)}||_{2,\gamma}&=&||(P_{t}u_0)(x)+(\mathcal{A}u^{(n)})(t,x)-(P_{t}u_0)(x)-(\mathcal{A}u^{(n-1)})(t,x)||_{2,\gamma}\\
 &\leq&||(\mathcal{A}u^{(n)})(t,x)-(\mathcal{A}u^{(n-1)})(t,x)||^{2,\gamma}\\
 &\leq& C^{**}||u^{(n)}-u^{(n-1)}||_{2,\gamma}\\
 &\leq&\frac{1}{4}||u^{(n)}-u^{(n-1)}||_{2,\gamma}\\
 &\leq& (\frac{1}{4})^n ||u^{(1)}-u^{(0)}||_{2,\gamma}\\
 &\leq&(\frac{1}{4})^n 2L_1.
\end{eqnarray*}
For any $m\geq n>0$,
$$
 ||u^{(m)}-u^{(n)}||_{2,\gamma}\leq (\frac{1}{4})^n \frac{1-(\frac{1}{4})^{m-n-1}}{1-\frac{1}{4}} 2L_1.
$$
so $||u^{(m)}-u^{(n)}||_{2,\gamma} \to 0 $ as $ n\to \infty$. Since $\mathcal{L}^{\gamma,2}$ is complete, then there exists  $u\in \mathcal{L}^{\gamma,2}$, such that $u^{(n)}\to u$ in the norm sense.

The uniqueness of the the solution to Equation \eqref{equcompen} follows easily the above argument by picking $u$ and $v$ as two solutions to  the equation, and by using Corollary \ref{stoyoungcompen1}.
\end{proof}

\begin{proof}[Proof of Theorem \ref{intermittency-compen}]
 We begin with the isometry \eqref{isometry},
 $$
 \E|u(t,x)|^2=|(P_t u_0)(x)|^2+\int_0^t\int_{\R^d}\int_{\R^d} |G(t-s,y-x)|^2\E|\sigma(u(s,y),h)|^2\mu(\mathrm{d}h)\mathrm{d}y\mathrm{d}s.
 $$
By assumption,  Condition \ref{cond-compen}, and Lemma \ref{Lem:Green1} we have
 \begin{eqnarray*}
 \E|u(t,x)|^2&\leq&\sup_{x\in \R^d}|(P_t u_0)(x)|^2+\int_0^t\int_{\R^d}\int_{\R^d} |G(t-s,y-x)|^2\E|\sigma(u(s,y),h)|^2\mu(\mathrm{d}h)\mathrm{d}y\mathrm{d}s\\
 &\leq&\sup_{x \in \R^d}|u_0(x)|^2+K L^2\int_0^t \int_{\R^d} |G(t-s,y-x)|^2\sup_{y\in R^d}\E|u(s,y)|^2 \mathrm{d}y\mathrm{d}s\\
 &\leq&\eta_1^2+K C^*L^2\int_0^t \sup_{y\in \R^d}\E|u(s,y)|^2(t-s)^{-\beta d/\alpha} \mathrm{d}y\mathrm{d}s\\
 \end{eqnarray*}
By letting $F(t)=\sup_{x\in {\R^d}} \E|u(t,x)|^2$, the above inequality becomes:
 \begin{eqnarray*}
 F(t)&\leq&\eta_1^2+K C^*L^2\int_0^t\frac{ F(s)}{(t-s)^{\beta d/\alpha}}\mathrm{d}y\mathrm{d}s\\
\end{eqnarray*}

Since $\beta d/\alpha <1$, then $1-(\beta d/\alpha)>0$. By setting $c_1 =\eta_1^2,\kappa^{'}=K C^{*} L^2, \rho=1-(\beta d/\alpha)$ in Proposition \ref{expestimateupper}, we obtain the required result.\\

{

Now let us move on to the proof of the second part  of the theorem.

Similarly, by Condition \ref{condof-beyondlinear}, and Lemma \ref{Lem:Green1} we have
 \begin{eqnarray*}
 \E|u(t,x)|^2&\geq&\inf_{x\in \R^d}|(P_t u_0)(x)|^2+\int_0^t\int_{\R^d}\int_{\R^d} |G(t-s,y-x)|^2\E|\sigma(u(s,y),h)|^2\mu(\mathrm{d}h)\mathrm{d}y\mathrm{d}s\\
 &\geq&\inf_{x \in \R^d}|u_0(x)|^2+\kappa L^2\int_0^t \int_{\R^d} |G(t-s,y-x)|^2\inf_{y\in \R^d}\E|u(s,y)|^2 \mathrm{d}y\mathrm{d}s\\
 &\geq&\eta_2^2+\kappa C^*L^2\int_0^t \inf_{y\in \R^d}\E|u(s,y)|^2(t-s)^{-\beta d/\alpha} \mathrm{d}y\mathrm{d}s\\
 \end{eqnarray*}
By letting $F(t)=\inf_{x\in {\R^d}} \E|u(t,x)|^2$, it yields
 \begin{eqnarray*}
 F(t)&\geq&\eta_2^2+\kappa C^*L^2\int_0^t\frac{ F(s)}{(t-s)^{\beta d/\alpha}}\mathrm{d}y\mathrm{d}s\\
\end{eqnarray*}

By Proposition \ref{expestimate}, we obtain the required result for the lower bound for all $t>0$.\\

}

\end{proof}
{
\begin{proof}[Proof of Theorem \ref{lowerbound}]
Following Theorem  \ref{Existcompen}, there exists a unique mild solution to Equation \ref{equcompen} when $\beta=1$
, that is
\begin{eqnarray*}
u(t,x)&=&(P_tu_0)(x)+\int_0^t\int_{\R^d}\int_{\R^d}\sigma(u(s,y),h)G(t-s,y-x)\tilde{N}(\mathrm{d}h,\mathrm{d}y,\mathrm{d}s).\\
\end{eqnarray*}
Take the second moment together with Condition \ref{condof-beyondlinear}, we have
\begin{eqnarray*}
E|u(t,x)|^2&=&|(P_tu_0(x))|^2+\int_0^t\int_{\R^d}\int_{\R^d} E|\sigma(u(s,y),h)|^2|G(t-s,y-x)|^2 \nu(\mathrm{d}h)\mathrm{d}y\mathrm{d}s\\
&\geq &|(P_tu_0(x))|^2+L^2\int_0^t\int_{\R^d}\int_{\R^d} \overline{J}(h)^2E|u(s,y)|^2|G(t-s,y-x)|^2 \mu(\mathrm{d}h)\mathrm{d}y\mathrm{d}s\\
&\geq &|(P_tu_0(x))|^2+\kappa L^2\int_0^t\int_{\R^d} E|\sigma(u(s,y),h)|^2|G(t-s,y-x)|^2\mathrm{d}y\mathrm{d}s\\
\end{eqnarray*}
Set $\eta:=\inf_{x\in \R^d} u_0(x)$ and $\tilde{t}=t-s$, we obtain
\begin{eqnarray*}
&&\int_0^\infty e^{-\lambda  t}E|u(t,x)|^2 \mathrm{d}t\\
&\geq &\int_0^\infty e^{-\lambda t}|(P_tu_0(x))|^2\mathrm{d}t+\kappa L^2\int_0^\infty e^{-\lambda t}\int_0^t\int_{\R^d} E|\sigma(u(s,y),h)|^2|G(t-s,y-x)|^2\mathrm{d}y\mathrm{d}s\mathrm {d}t\\
&=&\int_0^\infty e^{-\lambda t}|(P_tu_0(x))|^2\mathrm{d}t+\kappa L^2\int_{\R^d} \int_0^\infty \int_0^\infty e^{-\lambda (\tilde{t}+s)}|G(\tilde{t},y-x)|^2 \mathrm{d}\tilde{t} E|u(s,y)|^2 \mathrm{d}s \mathrm{d}y\\
&=&\int_0^\infty e^{-\lambda t}|(P_tu_0(x))|^2\mathrm{d}t+\kappa L^2\int_{\R^d} \int_0^\infty e^{-\lambda \tilde{t}}|G(\tilde{t},y-x)|^2 \mathrm{d}\tilde{t} \int_0^\infty e^{-\lambda s}E|u(s,y)|^2 \mathrm{d}s \mathrm{d}y
\end{eqnarray*}

Let $F_\lambda(x):=\int_0^\infty e^{-\lambda t}E|u(t,x)|^2\mathrm{d}t,$ then we obtain:

$$F_\lambda(x)\geq G_\lambda(x)+\kappa L^2 (F_\lambda *H_\lambda )(x)$$
Where $$G_\lambda (x)=\int_0^\infty  e^{-\lambda  t}|(P_tu_0(x))|^2dt$$
and
$$
H_\lambda (x)= \int_0^\infty  e^{-\lambda  t}|p(t,x)|^2dt,
$$
{where $p(t,x)$ is the density of the $\alpha$-stable process whose generator is the $-(-\Delta)^{\alpha/2}$, since $\beta=1.$\\}
\color{black}
Then by using the same line of ideas as in the proof of Theorem 2.7 in \cite{foondun-khoshnevisan-09} we can show that
$$
F_\lambda (x)\geq \eta^2\lambda^{-1} \sum _{n=0} ^\infty(\kappa L^2\Upsilon(\lambda ))^n.
$$
Form this we see that $ F_\lambda(x)=\infty$ as long as  $\kappa L^2\Upsilon(\lambda )\geq 1.$
\end{proof}
}

\color{black}
\begin{proof}[Proof of Theorem \ref{nonexistcompen}]
\begin{eqnarray*}
\E|u(t,x)|^2=|P_t u_0(x))|^2+\int_0^t\int_{\R^d} \int_{\R^d} |G_{t-s}(y-x)|^2 \E|\sigma(u(s,y),h)|^2\mu(\mathrm{d}h)\mathrm{d}y\mathrm{d}s
\end{eqnarray*}
Let $\eta:=\inf_{x \in {\R^d}}u_0(x)$, by Condition \ref{condoflip} and Jensen's inequality, we have

\begin{eqnarray*}
\E|u(t,x)|^2&\geq& \eta^2+\int_0^t\int_{\R^d} \int_{\R^d} |G_{t-s}(y-x)|^2 \E|\sigma(u(s,y),h)|^2\mu(\mathrm{d}h)\mathrm{d}y\mathrm{d}s\\
&\geq& \eta^2+L^2 \int_0^t\int_{\R^d} \int_{\R^d} \overline{J}(h)^2 |G_{t-s}(y-x)|^2 \E|u(s,y)|^{2\rho}\mu(\mathrm{d}h)\mathrm{d}y\mathrm{d}s\\
&\geq& \eta^2+\kappa L^2 \int_0^t\int_{\R^d} |G_{t-s}(y-x)|^2 (\E|u(s,y)|^2)^\rho \mathrm{d}y\mathrm{d}s\\
&\geq&\eta^2+\kappa L^2 \int_0^t\int_{\R^d} |G_{t-s}(y-x)|^2 \inf_{y \in \R^d}(\E|u(s,y)|^2)^\rho \mathrm{d}y\mathrm{d}s\\
\end{eqnarray*}
Let $I(t):=\inf_{y\in {\R^d}}\E|u(s,y)|^2$, by Lemma \ref{Lem:Green1}, we have

\begin{eqnarray*}
I(t)&\geq&\eta^2+\kappa L^2 \int_0^t\int_{\R^d} |G_{t-s}(y-x)|^2 [I(s)]^\rho\mathrm{d}y\mathrm{d}s\\
&\geq& \eta^2+\kappa L^2 \int_0^t C^*(t-s)^{-\beta/\alpha} [I(s)]^\rho\mathrm{d}s\\
\end{eqnarray*}

Define the Laplace transform of $I$ as
$$ \tilde{I}(\theta):=\int_0^\infty e^{-\theta s}I(s)ds.$$

It is easy to see that
 $$
 \int_0^\infty e^{-\theta s}||G_{s}||^2_{L^2(\rd)}ds=\int_0^\infty e^{-\theta s}C^*s^{-\beta d/\alpha} ds=C_1 \theta^{-(1-\beta d/\alpha)}
 $$
 for all $\theta>0$, where $C_1=C^* \Gamma(1-\beta d/\alpha)$.
 It follows that
\begin{eqnarray*}
\tilde{I}(\theta)&\geq&\frac{\eta^2}{\theta}+\kappa L^2 \int^\infty_0 e^{-\theta t}\int_0^t (t-s)^{-\beta d/\alpha} [I(s)]^\rho \mathrm{d}s \mathrm{d}t\\
&\geq&\frac{\eta^2}{\theta}+\kappa L^2 \int^\infty_0 e^{-\theta s}\bigg[\int_s^\infty C^{*} e^{-\theta(t-s)}(t-s)^{-\beta d/\alpha} \mathrm{d}t\bigg][I(s)]^\rho \mathrm{d}s \\
&\geq& \frac{\eta^2}{\theta}+\kappa L^2 C_1 \theta^{-(1- \beta d/\alpha)} \int_0^\infty e^{-\theta s}[I(s)]^\rho ds.
\end{eqnarray*}
  Multiply both sides by $\theta$ and use Jensen's inequality to get

  \begin{eqnarray*}
  \theta\tilde{I}(\theta)&\geq& \eta^2+\kappa L^2 C_1 \theta^{-(1-\beta d/\alpha)}\bigg[ \int_0^\infty \theta e^{-\theta s}[I(s)]ds\bigg]^\rho\\
  &=& \eta^2+\kappa L^2 C_1 \theta^{-(1-\beta d/\alpha)}\big[\theta \tilde{I}(\theta)\big]^\rho,
  \end{eqnarray*}
  for all $\theta >0$. It follows that $ \theta\tilde{I}(\theta)\geq \eta^2>0$, and hence  $ \theta\tilde{I}(\theta)>0$ for all $\theta >0.$ Moreover,
   $$ [\theta\tilde{I}(\theta)]^{\rho}=[\theta\tilde{I}(\theta)]^{\rho-1}\theta\tilde{I}(\theta)\geq \eta^{2\rho-2}\theta\tilde{I}(\theta).$$
Hence,
 $$ \theta\tilde{I}(\theta)\geq \eta^2+\kappa L^2 C_1\theta^{-(1-\beta d/\alpha)}\eta^{2\rho-2}\theta\tilde{I}(\theta).$$
For $0<\theta<\theta_0:=(\kappa L^2 C_1\eta^{2\rho-2})^{1-\beta d/\alpha}$, we will show that $\tilde{I}(\theta)=\infty.$

Under the assumption that $0<\theta\leq\theta_0$, the constant $$A:=\kappa L^2 C_1\theta^{-(1-\beta d/\alpha)}\eta^{2\rho-2} \geq 1.$$
With the recursive argument, we have
$$\theta\tilde{I}(\theta)\geq \eta^2+A^n\theta\tilde{I}(\theta) $$ for any positive integer $n$. So $\theta\tilde{I}(\theta)=\infty,$ which means $ \tilde{I}(\theta)=\infty$ for a specified range of $\theta$.
Since ${I}(\theta_0)=\infty $, then $I(t)\geq ce^{\theta_0 t}$ for large $t$.
 So for $t$ large enough, we have
\begin{eqnarray*}
 I(t)&\geq& \eta^2+\kappa L^2 \int_0^t C^*(t-s)^{-\beta d/\alpha} [I(s)]^\rho\mathrm{d}s\\
&\geq&\eta^2+\kappa C^*L^2 t^{-\beta d/\alpha}\int_{t/2}^t e^{{\theta_0}\rho s} \mathrm{d}s\\
&\geq&\eta^2+\kappa C^*L^2 t^{-\beta d/\alpha} e^{{\theta_0}\rho t} (1-e^{-{\theta_0}\rho t/2})\\
\end{eqnarray*}
 Since $\theta_0\rho t/2>0$, using the inequality $1-e^{-x}> \frac{x}{x+1}$ for $x>-1$, we have
 $$1-e^{-{\theta_0}\rho t/2}>\frac{\theta_0\rho t/2}{1+\theta_0\rho t/2}.$$
 And there exists a $t^{*}$, such that $\frac{\theta_0\rho t/2}{1+\theta_0\rho t/2}>\frac{\theta_0\rho t/2}{2}.$
 Hence,
\begin{eqnarray*}
 I(t)&\geq& \eta^2+\kappa C^*L^2 t^{-\beta d/\alpha} e^{{\theta_0}\rho t} (1-e^{-{\theta_0}\rho t/2})\\
 &\geq& \eta^2+\kappa C^*L^2 t^{-\beta d/\alpha} e^{{\theta_0}\rho t}\frac{\theta_0\rho t/2}{1+\theta_0\rho t/2}\\
  &\geq& \eta^2+\kappa C^*L^2 t^{-\beta d/\alpha} e^{{\theta_0}\rho t}\frac{\theta_0\rho t/2}{2}\\
 &\geq& \eta^2+\frac{\kappa\rho \theta_0 C^*L^2}{4}t^{1-\beta d/\alpha} e^{{\theta_0}\rho t}.\\
\end{eqnarray*}
for $t$ large enough. Let $\theta_1=\theta_0 \rho$, then we have $\tilde{I}(\theta)=\infty.$ Otherwise, there is no $\theta$ less than $\theta_1$ such that $\tilde{I}(\theta)<\infty.$

On the other hand, if we assume there is a finite energy solution, we have $\tilde{I}(\theta_{*})<\infty $ for some $\theta_{*}>0$. Hence $$\tilde{I}(\theta)= \int_0^\infty e^{-\theta s}I(s)ds=\int_0^\infty e^{-(\theta-\theta_\ast + \theta_\ast)s}I(s)ds\leq \int_0^\infty e^{-\theta_\ast s}I(s)ds<\infty,$$ for $\theta\geq\theta_\ast$. That means $\tilde{I}(\theta)<\infty$ for all $\theta \geq \theta_\ast.$ But this contradicts the above argument. Therefore there is no finite energy solution.
\end{proof}

\begin{proof}[Proof of Theorem \ref{nonexistofcompen}]
 We begin with the isometry \eqref{isometrynon},
 $$
 \E|u(t,x)|^2=|(P_t u_0)(x)|^2+\int_0^t\int_{\R^d}\int_{\R^d} |G(t-s,y-x)|^2\E|\sigma(u(s,y),h)|^2\mu(\mathrm{d}h)\mathrm{d}y\mathrm{d}s.
 $$
Let $\eta:=\inf_{x\in {\R^d}}u_0(x)$. By Condition \ref{condoflip} and Jensen's inequality, we have
\begin{eqnarray*}
 \E|u(t,x)|^2&\geq& \eta^2+\kappa L^2\int_0^t\int_{\R^d} |G(t-s,y-x)|^2\E(u(s,y))|^{2\rho}\mathrm{d}y\mathrm{d}s\\
 &\geq& \eta^2+\kappa L^2\int_0^t \int_{\R^d} |G(t-s,y-x)|^2(\inf_{y\in {\R^d}}\E|u(s,y)|^2)^{\rho} \mathrm{d}y\mathrm{d}s\\
\end{eqnarray*}
 Let $F(t)=\inf_{x\in {\R^d}} \E|u(t,x)|^2$, then the above is reduced to be
 $$F(t)\geq \eta^2+\kappa L^2\int_0^t\int_{\R^d} |G(t-s,y-x)|^2 F^\rho(s)\mathrm{d}y\mathrm{d}s.$$
 Using  Lemma \ref{Lem:Green1} yields
 \begin{eqnarray*}
 F(t)&\geq& \eta^2+\kappa L^2\int_0^t\int_{\R^d} |G(t-s,y-x)|^2 F^\rho(s)\mathrm{d}y\mathrm{d}s \\
 &=& \eta^2+\kappa L^2\int_0^t C^*(t-s)^{-\beta d/\alpha}F^\rho(s)\mathrm{d}s\\
\end{eqnarray*}

Since $\eta^2$,$\kappa C^*L^2>0$, and $\rho>1,\beta d/\alpha>0$, then from Proposition \ref{nonlinearinequ}, we know that there exists a $t_0$ such that $\E|u(t,x)|^2=\infty$ for all
$t\geq t_0$, which means
$$\sup_{0\leq t\leq T}\sup_{x\in {\R^d}} \E|u(t,x)|^2=\infty  \text{ for any }T\geq t_0.$$

So $u(t,x)$ will blow up in finite time and there is no random field solution to equation \eqref{equcompen}.
\end{proof}

\section{Proof of results for the Poisson noise}\label{sec:6}

\begin{proof}[Proof of Theorem \ref{existofnoncompen}]
We use Picard iteration to show existence of solutions.
Let $u^{(0)}
(t,x) := u_0(x)$, and for all $n\geq 0$
$$
u^{(n+1)}(t,x)=(P_{t}u_0)(x)+(\mathcal{B}u^{(n)})(t,x)
$$
and
$$
(P_{t}u_0) (x)=\int_{\R^d} P_t (x-y) u_0 (y)dy.
$$
Then we have
$$
||P_{t}u_0 (x)||_{1,\gamma}=\{\sup_{t>0}\sup_{x\in R^d} e^{-\gamma t}\E|u_0(x)|, \gamma>0\} \leq \sup _{x\in \R^d}|u_0 (x)|.
$$

By Propostion \ref{stochyoungnoncompen}, $||\mathcal{B}u||_{1,\gamma} \leq \frac{K Lip_{\sigma}}{\gamma} ||u||_{1,\gamma}\  \hbox{for } \ \gamma >0$. This  yields

$$||u^{(n+1)}||_{1,\gamma} \leq \sup_{x \in \R^d}|u_0(x)|+\frac{K Lip_{\sigma}}{\gamma} ||u^{(n)}||_{1,\gamma}.$$

We can find a constant $L_3$-- depending only on $Lip_\sigma,K$--such that  $\gamma:=L_3 KLip_{\sigma}$ satisfying

$$\frac{K Lip_{\sigma}}{\gamma}\leq\frac{1}{4},$$

and let $\theta=\sup_{x\in \R^d}|u_0(x)|$.

Then
\begin{eqnarray*}
||u^{(n+1)}||_{1,\gamma}&=&||(P_{t}u_0)(x)+(\mathcal{B}u^{(n)})(t,x)||_{1,\gamma}\\
&\leq& \sup_{x\in \R^d}|u_0(x)|+\frac{1}{4}||u^{(n)}||_{1,\gamma} \\
&\leq& \theta+\frac{1}{4}(\theta+||u^{(n)}||_{1,\gamma})\\
&\leq& \sum_{j=0}^{n} \frac{\theta}{4^j} +\frac{1}{4^{n+1}}\sup_{x\in \R^d}|u^{(0)}|\\
&=&\frac{\theta}{1-\frac{1}{4}}+\frac{1}{4^{n+1}}\theta\\
&\leq& \frac{4}{3}\theta+\theta=\frac{7}{3}\theta=L_1.\\
\end{eqnarray*}

so we have $||u^{(n+1)}||_{1,\gamma} \in\cup_{\gamma>0} \mathcal{L}^{\gamma,1}.$
i.e.
$$\sup_{t>0}\sup_{x\in \R^d} e^{-\gamma t}\E|u^{(n+1)}(t,x)|\leq L_4.$$
Hence $\sup_{x\in \R^d} \E|u^{(n+1)}(t,x)|\leq L_4 e^{\gamma t}, t>0.$

Next, we use the Banach fixed point theorem to show the existence of the solution.
By the Corollary \ref{stochyoungnoncompen2}, we have
\begin{eqnarray*}
 ||u^{(n+1)}-u^{(n)}||_{1,\gamma}&=&||(P_{t}u_0)(x)+(\mathcal{B}u^{(n)})(t,x)-(P_{t}u_0)(x)-(\mathcal{B}u^{(n-1)})(t,x)||_{1,\gamma}\\
 &\leq&||(\mathcal{B}u^{(n)})(t,x)-(\mathcal{B}u^{(n-1)})(t,x)||^{1,\gamma}\\
 &\leq& C^{'}||u^{(n)}-u^{(n-1)}||_{1,\gamma}\\
 &\leq&\frac{1}{4}||u^{(n)}-u^{(n-1)}||_{1,\gamma}.\\
\end{eqnarray*}
Following the last part of the proof of Theorem \ref{Existcompen}, we obtain the existence and uniqueness of the solution to equation \eqref{equnoncompen}.
\end{proof}

\begin{proof}[Proof of Theorem \ref{intermittency-noncomp}] We use  Walsh isometry \eqref{isometrynon}
\begin{eqnarray*}
&\E| u(t,x)|=(P_t u_0)(x)+\int_0^t\int_{\R^d}\int_{\R^d} G(t-s,y-x)\E|\sigma(u(s,y),h)|\mu(\mathrm{d}h)\mathrm{d}y\mathrm{d}s
\end{eqnarray*}
Let $ \eta:=\inf_{x \in \R^d} u_0(x)$. By Condition \ref{condofnoncompen1},  we have
\begin{eqnarray*}
\E |u(t,x)|&=&(P_t u_0)(x)+\int_0^t\int_{\R^d}\int_{\R^d} G(t-s,y-x)\E|\sigma(u(s,y),h)|\mu(\mathrm{d}h)\mathrm{d}y\mathrm{d}s\\
&\geq& \eta+\kappa L \int_0^t\int_{\R^d} G(t-s,y-x)\inf_{y \in \R^d}\E |u(s,y)|\mathrm{d}y\mathrm{d}s\\
&\geq& \eta+\kappa L \int_0^t\inf_{y \in \R^d}\E |u(s,y)| \bigg[\int_{\R^d}G(t-s,y-x) \mathrm{d}y\bigg]\mathrm{d}s\\
&\geq& \eta+\kappa C^* L \int_0^t\inf_{y \in \R^d}\E |u(s,y)|\mathrm{d}s\\
\end{eqnarray*}
Let $f(t):=\inf_{x \in \R^d} \E |u(t,x)|$, the above yields
\begin{eqnarray*}
f(t) \geq \eta+\kappa C^* L \int_0^t {f(s)}{} \mathrm{d}s\\
\end{eqnarray*}
So we get the intended result by Proposition \ref{expestimate}.
\end{proof}

\begin{proof}[Proof of Theorem \ref{no-rf-sol-noncom}]
\begin{equation}
\begin{split}
&\E| u(t,x)|=(P_t u_0)(x)+\int_0^t\int_{\R^d}\int_{\R^d} G(t-s,y-x)\E|\sigma(u(s,y),h)|\mu(\mathrm{d}h)\mathrm{d}y\mathrm{d}s
\end{split}
\end{equation}
Let $ \eta:=\inf_{x \in \R^d} u_0(x)$ and $\eta>0$, we have $(P_t u_0)(x)>\eta$. By Condition \ref{condoflipnoncompen2} and Jensen's inequality,  we have
\begin{eqnarray*}
\E|u(t,x)|&=&(P_t u_0)(x)+\int_0^t\int_{\R^d}\int_{\R^d} G(t-s,y-x)\E|\sigma(u(s,y),h)|\mu(\mathrm{d}h)\mathrm{d}y\mathrm{d}s\\
&\geq& \eta+\kappa L \int_0^t\int_{\R^d} G(t-s,y-x)\inf_{y \in \R^d}\E |u(s,y)|^\lambda \mathrm{d}y\mathrm{d}s\\
&\geq& \eta+\kappa L \int_0^t\inf_{y \in \R^d}\E |u(s,y)|^\lambda \bigg[\int_{\R^d}G(t-s,y-x) \mathrm{d}y\bigg]\mathrm{d}s\\
&\geq& \eta+\kappa C^* L \int_0^t\inf_{y \in \R^d}\E |u(s,y)|^\lambda  \mathrm{d}s.
\end{eqnarray*}

Let $f(t):=\inf_{x \in \R^d} \E |u(t,x)|$,  then we have
$$f(t)\geq \eta+\kappa C^* L \int_0^t  {f(s)^\lambda}{} \mathrm{d}s.$$
Since $\eta$, $\kappa C^*L , \beta d/\alpha>0$, and $\lambda>1$, then  there is no random field solution to Equation \eqref{equnoncompen}.

\end{proof}



\end{document}